%% file: main.tex
\documentclass{article}
\input{preamble}

\title{Locally indicable groups outside \texorpdfstring{Brodski\u{\i}'s}{Brodskii's} quasivariety}
\author{Hannah Bo\ss}
\date{\today}

\begin{document}
	\maketitle

\begin{abstract}
    We provide a set of group properties that together obstruct membership in the quasivariety that Brodski\u{\i} constructed in his proof that torsion-free one-relator groups are locally indicable. There are examples of locally indicable groups with these properties, which implies that Brodski\u{\i}'s quasivariety is strictly contained in the quasivariety of all locally indicable groups. This disproves a conjecture of Howie.
\end{abstract}
	
\section{Introduction}
A group is called \emph{locally indicable} if each of its non-trivial finitely generated subgroups admits an epimorphism onto $\Z$.
In \cite{BrodskiiAnnouncement}, S. D. Brodski\u{\i} announced the first proof of the result that torsion-free one-relator groups are locally indicable. His proof was published four years later in \cite{Brodskii}. Independent proofs of the same result were given by J. Howie around the same time in \cite{Howie_OLIG} and later in \cite{Howie_ASPOATOB}.

The idea of Brodski\u{\i}'s algebraic proof is to show that a certain quasivariety of groups contains all free groups and is closed under taking one-relator products whose relator is not a proper power. A \emph{quasivariety} of groups is a non-empty class of groups that is closed under the operations of taking subgroups, arbitrary direct products and ultraproducts.
All classes of groups considered in this article are meant to be \emph{abstract} in the sense that membership of a group $G$ in a class $\mathcal{C}$, which we denote by $G \in \mathcal{C}$, only depends on the isomorphism type of $G$.
The quasivariety in question is $\frakAt$, the smallest non-trivial quasivariety of groups that is closed under group extensions. Because the class of all locally indicable groups is a non-trivial quasivariety closed under group extensions, the class $\frakAt$ is contained in the class of locally indicable groups.

In \cite[p.~62]{Howie_HTGO-RGT}, Howie suggests that the inclusion of these two quasivarieties could actually be an equality. In this note, we disprove this conjecture by providing a set of properties that jointly imply that a group is not contained in $\frakAt$ and giving examples of locally indicable groups satisfying these properties.

The author of this paper would like to thank her supervisor Giles Gardam for encouraging her to think about Howie's conjecture, Jim Howie for the helpful conversation on the history of the subject that we had in June 2026, and Corentin Bodart and Francesco Matucci for explanations on generalisations of Thompson's group $F$.

This work was conducted without using artificial intelligence and funded by the Deutsche Forschungsgemeinschaft (DFG, German Research Foundation) under Germany's Excellence Strategy - EXC-2047/2 - 390685813 and by the European Union (ERC, SATURN, 101076148).


\section{\texorpdfstring{Brodski\u{\i}'s}{Brodskii's} quasivariety\texorpdfstring{ $\frakAt$}{}}

The following definition is due to Brodski\u{\i}.
\begin{definition}[\cite{Brodskii}]
Let $\frakA$ be the quasivariety of all torsion-free abelian groups. For each ordinal $\alpha$, define the class $\frakA_\alpha$ as follows:
\begin{itemize}
	\item $\frakA_0 = \frakA$,
	\item if $\alpha = \beta +1$ is a successor ordinal, then $\frakA_\alpha$ consists of all group extensions of groups from $\frakA_\beta$ by groups from $\frakA_\beta$,
	\item if $\alpha$ is a limit ordinal, then $\frakA_\alpha$ is given by all groups that embed into an ultraproduct of groups from $\bigcup_{\beta < \alpha} \frakA_\beta$.
\end{itemize}
\end{definition}

This is an ascending family of quasivarieties of groups (see \cite[Proof of Lemma 5]{Brodskii}), which stabilises after countably many steps:
\begin{lemma}[{\cite[Lemma 5]{Brodskii}}] There exists a countable ordinal $\tilde{\alpha}$ such that $\frakA_{\tilde{\alpha}}$ is the smallest non-trivial quasivariety of groups that is closed under group extensions.
\end{lemma}

Then by induction it follows that $\frakA_\alpha = \frakA_{\tilde{\alpha}}$ for all ordinals $\alpha$ with $\alpha > \tilde{\alpha}$. We denote $\frakA_{\tilde{\alpha}}$ by $\frakAt$ and call it \emph{Brodski\u{\i}'s quasivariety}.

Since the class of all locally indicable groups is a non-trivial quasivariety and is closed under group extensions, any group in $\frakAt$ is locally indicable. Howie asked in \cite{Howie_HTGO-RGT} whether the converse is true. Our main theorem is that this is not the case.

\begin{theorem} \label{Main theorem}
	There exist locally indicable groups that are not contained in Brodski\u{\i}'s quasivariety $\frakAt$.
\end{theorem}

To prove this, we first establish a criterion to deduce non-membership in Brodski\u{\i}'s quasivariety. There will be examples of locally indicable groups satisfying that criterion, which will imply Theorem \ref{Main theorem}.


\section{A criterion for lying outside \texorpdfstring{$\frakAt$}{Brodskii's quasivariety}}

\begin{theorem} \label{Theorem criterion for not in frakAt}
    Let $G$ be a group satisfying the following properties:
    \begin{enumerate}[label=(\arabic*)]
        \item \label{Item ass fp} $G$ is finitely presented.
        \item \label{Item ass not abelian} $G$ is not abelian.
        \item \label{Item ass quotients} Every proper quotient of $G$ is abelian.
        \item \label{Item ass embed in comm subgroup} $G$ is isomorphic to a subgroup of its own commutator subgroup $G^\prime$.
    \end{enumerate}
    Then $G$ is not contained in $\frakAt$.
\end{theorem}

Property \ref{Item ass quotients} is for instance satisfied whenever the centre $\Zen (G)$ of $G$ is trivial and $G^\prime$ is simple. The examples that we will mention now have property \ref{Item ass quotients} for this reason.

Theorem \ref{Theorem criterion for not in frakAt} implies Theorem \ref{Main theorem} as follows.

\begin{proof}[Proof of Theorem \ref{Main theorem}]
    Thompson's group $F$ is locally indicable, in fact bi-orderable (see for example \cite[Theorem 4.11]{CannFloyParr_INORTG}). Moreover, $F$ satisfies the properties from Theorem \ref{Theorem criterion for not in frakAt}, which follows for instance from \cite[Theorem 3.4, (Proof of) Theorem 4.3, Theorem 4.1 and Lemma 4.4]{CannFloyParr_INORTG}. Hence, Theorem \ref{Theorem criterion for not in frakAt} implies that $F$ is a locally indicable group that is not contained in the quasivariety $\frakAt$.
\end{proof}

There are more examples of locally indicable groups satisfying the conditions of Theorem \ref{Theorem criterion for not in frakAt} among the generalisations of Thompson's group $F$ that are extensively studied in \cite{BierStreb} and denoted $G(I;A,P)$ there.

To prove Theorem \ref{Theorem criterion for not in frakAt}, we make use of a combination of two results from \cite{CSC}.

\begin{theorem}[{\cite[Theorem 7.2.5 and Corollary 7.1.21]{CSC}}] \label{Theorem locally embed implies residually} Let $\mathcal{C}$ be a class of groups which is closed under taking subgroups. Then any finitely presented group which embeds into an ultraproduct of groups from $\mathcal{C}$ is residually $\mathcal{C}$.
\end{theorem}

\begin{remark}
    While the above implication holds for finitely presented groups only, the converse implication is true for any (non-necessarily finitely presented) group provided that the class $\mathcal{C}$ is closed under finite direct products (cf. \cite[Theorem 7.2.5 and Corollary 7.1.14]{CSC}).
\end{remark}

\begin{proof}[Proof of Theorem \ref{Theorem criterion for not in frakAt}]
	We proceed by contradiction. Assume that $G \in \frakAt$.
    Then there exists a minimal ordinal $\alpha$ such that $G \in \frakA_\alpha$. Since $G$ is not abelian, $\alpha > 0$.
	
	Assume first that $\alpha = \beta +1 $ is a successor ordinal. Then there are groups $N,Q \in \frakA_\beta$ such that
	\begin{equation*}
		1 \to N \to G \to Q \to 1
	\end{equation*}
	is a short exact sequence of groups. Then by assumption \ref{Item ass quotients} on $G$, either $Q$ is isomorphic to $G$ or $Q$ is abelian. The first case is not possible since this would mean $G\in \frakA_\beta$, contradicting the minimality of $\alpha$. So $Q$ must be abelian. This means that $N$ contains the commutator subgroup $G^\prime$, which itself contains a subgroup isomorphic to $G$ by property \ref{Item ass embed in comm subgroup}. Since $\frakA_\beta$ is a quasivariety and therefore closed under taking subgroups, we again obtain $G \in \frakA_\beta$, which contradicts the choice of $\alpha$.
    
	So $\alpha$ must be a limit ordinal. This means that $G$ embeds into an ultraproduct of groups from $\bigcup_{\beta < \alpha} \frakA_\beta$. This class is closed under taking subgroups since every $\frakA_\beta$ is. Since $G$ is finitely presented by assumption \ref{Item ass fp}, Theorem \ref{Theorem locally embed implies residually} implies that $G$ is residually $\bigcup_{\beta < \alpha} \frakA_\beta$.
    We show that this induces a contradiction.
    Let $g \in G^\prime \setminus \{1\}$ be a non-trivial element of the commutator subgroup of $G$, which exists by assumption \ref{Item ass not abelian}. Then there is an ordinal $\beta < \alpha$, a group $H \in \frakA_\beta$ and a surjective homomorphism $\varphi \colon G \twoheadrightarrow H$ such that $\varphi (g) \neq 1$. By assumption \ref{Item ass quotients}, either $\varphi$ is an isomorphism or $H$ is abelian. The first is not possible as it would mean $G \in \frakA_\beta$, contradicting the minimality of $\alpha$ as before. So $H$ is abelian, which contradicts the fact that $g \in G^\prime$ and $\varphi(g) \neq 1$.
	
    Therefore, the assumption was wrong and $G$ is not contained in the quasivariety $\frakAt$.
\end{proof}

	
	\bibliographystyle{alpha}
	\bibliography{references}
\end{document}

%% file: preamble.tex
\usepackage{amsmath, amssymb, amsfonts, amsthm}

\usepackage{mathtools}
\usepackage{scalerel} 
\usepackage{microtype}
\usepackage{enumitem}
\usepackage{hyperref}

\theoremstyle{plain}
\newtheorem{theorem}{Theorem}[section]
\newtheorem{lemma}[theorem]{Lemma}

\theoremstyle{definition}
\newtheorem{definition}[theorem]{Definition}

\theoremstyle{remark}
\newtheorem{remark}[theorem]{Remark}

\DeclareMathOperator{\Zen}{Z}

\newcommand{\Z}{\mathbb{Z}}

\newcommand{\frakA}{\mathfrak{A}}
\newcommand{\frakAt}{\tilde{\mathfrak{A}}}

%% file: references.bib
@book {BierStreb,
    AUTHOR = {Bieri, R. and Strebel, R.},
     TITLE = {On groups of {PL}-homeomorphisms of the real line},
    SERIES = {Mathematical Surveys and Monographs},
    VOLUME = {215},
 PUBLISHER = {American Mathematical Society, Providence, RI},
      YEAR = {2016},
     PAGES = {xvii+174},
      ISBN = {978-1-4704-2901-0},
   MRCLASS = {20F38 (20E32 20F05 20F28)},
       DOI = {10.1090/surv/215},
       URL = {https://doi.org/10.1090/surv/215},
}

@article{Brodskii,
    AUTHOR = {Brodski\u{\i}, S. D.},
     TITLE = {Equations over groups, and groups with one defining relation},
   JOURNAL = {Sibirsk. Mat. Zh.},
  FJOURNAL = {Akademiya Nauk SSSR. Sibirskoe Otdelenie. Sibirski\u i\
              Matematicheski\u i\ Zhurnal},
    VOLUME = {25},
      YEAR = {1984},
    NUMBER = {2},
     PAGES = {84--103},
      ISSN = {0037-4474},
   MRCLASS = {20E06 (20E25)},
  MRNUMBER = {741011},
MRREVIEWER = {James\ Howie},
}

@article{BrodskiiAnnouncement,
    AUTHOR = {Brodski\u{\i}, S. D.},
     TITLE = {Equations over groups and groups with one defining relation},
   JOURNAL = {Uspekhi Mat. Nauk},
  FJOURNAL = {Akademiya Nauk SSSR i Moskovskoe Matematicheskoe Obshchestvo.
              Uspekhi Matematicheskikh Nauk},
    VOLUME = {35},
      YEAR = {1980},
    NUMBER = {4(214)},
     PAGES = {183},
      ISSN = {0042-1316},
   MRCLASS = {20F05},
  MRNUMBER = {586195},
MRREVIEWER = {F.\ Haimo},
}

@article {CannFloyParr_INORTG,
    AUTHOR = {Cannon, J. W. and Floyd, W. J. and Parry, W. R.},
     TITLE = {Introductory notes on {R}ichard {T}hompson's groups},
   JOURNAL = {Enseign. Math. (2)},
  FJOURNAL = {L'Enseignement Math\'ematique. Revue Internationale. 2e
              S\'erie},
    VOLUME = {42},
      YEAR = {1996},
    NUMBER = {3-4},
     PAGES = {215--256},
      ISSN = {0013-8584},
   MRCLASS = {20F32 (57M07)},
MRREVIEWER = {Michael\ L.\ Mihalik},
}

@book {CSC,
    AUTHOR = {Ceccherini-Silberstein, T. and Coornaert, M.},
     TITLE = {Cellular automata and groups},
    SERIES = {Springer Monogr. Math.},
 PUBLISHER = {Springer-Verlag, Berlin},
      YEAR = {2010},
     PAGES = {xx+439},
      ISBN = {978-3-642-14033-4},
}

@article{Howie_ASPOATOB,
    AUTHOR = {Howie, J.},
     TITLE = {A short proof of a theorem of {B}rodski\u{\i}},
   JOURNAL = {Publ. Mat.},
  FJOURNAL = {Publicacions Matem\`atiques},
    VOLUME = {44},
      YEAR = {2000},
    NUMBER = {2},
     PAGES = {641--647},
      ISSN = {0214-1493,2014-4350},
   MRCLASS = {20F05 (57M07)},
       DOI = {10.5565/publmat\_44200\_13},
       URL = {https://doi.org/10.5565/publmat_44200_13},
}

@incollection {Howie_HTGO-RGT,
    AUTHOR = {Howie, J.},
     TITLE = {How to generalize one-relator group theory},
 BOOKTITLE = {Combinatorial group theory and topology ({A}lta, {U}tah,
              1984)},
    SERIES = {Ann. of Math. Stud.},
    VOLUME = {111},
     PAGES = {53--78},
 PUBLISHER = {Princeton Univ. Press, Princeton, NJ},
      YEAR = {1987},
      ISBN = {0-691-08409-2; 0-691-08410-6},
   MRCLASS = {20F05 (20-02 20E06)},
  MRNUMBER = {895609},
MRREVIEWER = {F.\ Levin},
}

@article{Howie_OLIG,
    AUTHOR = {Howie, J.},
     TITLE = {On locally indicable groups},
   JOURNAL = {Math. Z.},
  FJOURNAL = {Mathematische Zeitschrift},
    VOLUME = {180},
      YEAR = {1982},
    NUMBER = {4},
     PAGES = {445--461},
      ISSN = {0025-5874,1432-1823},
   MRCLASS = {20F05 (57M20)},
MRREVIEWER = {Johannes\ Huebschmann},
       DOI = {10.1007/BF01214717},
       URL = {https://doi.org/10.1007/BF01214717},
}
